\documentclass[11pt,reqno]{amsproc}
\usepackage[margin=1in]{geometry}
\usepackage{amsmath, amsthm, amssymb}
\usepackage[colorlinks=true, pdfstartview=FitV, linkcolor=blue,citecolor=blue, urlcolor=blue]{hyperref}
\usepackage[abbrev,lite]{amsrefs}
\usepackage{times}
\usepackage[usenames,dvipsnames]{color}

\def\eps{\varepsilon}

\def\d{{\rm d}}

\def\R {\mathbb{R}}
\def\u {\boldsymbol{u}}

\def\RR {{\mathcal R}}

\def\ZZ {{\mathbb Z}}

\def\TT {{\mathbb T}^2}

\DeclareMathOperator*{\esup}{ess\,sup}

\def\de{{\partial}}

\newtheorem{proposition}{Proposition}[section]
\newtheorem{theorem}[proposition]{Theorem}
\newtheorem{corollary}[proposition]{Corollary}
\newtheorem{lemma}[proposition]{Lemma}
\theoremstyle{definition}

\numberwithin{equation}{section}

\title[On the global regularity for supercritical SQG]{On the global regularity for the supercritical SQG equation}

\author[M. Coti Zelati and V. Vicol]{Michele Coti Zelati and Vlad Vicol}

\address{Department of Mathematics, University of Maryland, College Park, MD 20742, USA}
\email{micotize@umd.edu}

\address{Department of Mathematics, Princeton University, Princeton, NJ 08544, USA}
\email{vvicol@math.princeton.edu}



\begin{document}

\begin{abstract}
We consider the initial value problem for the fractionally dissipative quasi-geostrophic equation
\[
\partial_t \theta + \RR^\perp \theta \cdot \nabla \theta + \Lambda^\gamma \theta = 0, \qquad  \theta(\cdot,0) =\theta_0 
\]
on $\TT = [0,1]^2$, with $\gamma \in (0,1)$. The coefficient in front of the 
dissipative term $\Lambda^\gamma = (-\Delta)^{\gamma/2}$ is normalized to $1$. 
We show that given a smooth initial datum  with $\|\theta_0\|_{L^2}^{\gamma/2} \|\theta_0\|_{\dot{H}^2}^{1-\gamma/2} \leq R$, 
where  {\em $R$ is arbitrarily large}, there exists $\gamma_1 = \gamma_1(R)  \in (0,1)$ such 
that for $\gamma \geq \gamma_1$, the solution of the supercritical SQG equation with 
dissipation $\Lambda^\gamma$ does not blow up in finite time.  The main ingredient  in 
the proof is a new concise proof of eventual regularity for the supercritical SQG equation, 
that relies solely on nonlinear lower bounds for the fractional Laplacian and the maximum principle.
\hfill \today
\end{abstract}


\maketitle

\section{Introduction}
The supercritical dissipative surface quasi-geostrophic equation reads
\begin{equation}\tag{SQG$_\gamma$} \label{eq:SQG:gamma}
\begin{cases}
\de_t\theta +\u \cdot \nabla \theta+\Lambda^\gamma\theta=0,\\
\u = \RR^\perp \theta = \nabla^\perp \Lambda^{-1} \theta,\\
\theta(0)=\theta_0,
\end{cases}
\end{equation}
where $(x,t)\in \TT\times[0,\infty)$, and $\TT = [0,1]^2$. Throughout this paper we 
take $\gamma\in [\gamma_0,1)$, where $\gamma_0\in(0,1)$ is an arbitrarily small fixed value. The 
data and the solution have zero mean on $\TT$, and we write $\Lambda = (-\Delta)^{1/2}$.  

Any sufficiently regular solution to \eqref{eq:SQG:gamma} satisfies the $L^\infty$ maximum principle
\[
\|\theta(t)\|_{L^\infty}\leq \|\theta_0\|_{L^\infty} 
\]
for all $t\geq 0$. This is the strongest known a priori bound for solutions to \eqref{eq:SQG:gamma}. In fact one may show that the $L^\infty$ norm decays exponentially~\cites{CC04,CTV13}.
On the other hand the dissipative SQG equation has a 
natural scaling symmetry: if $\theta(x,t)$ is a $\TT$-periodic solution to \eqref{eq:SQG:gamma} with datum $\theta_0(x)$,  then $\theta_\lambda(x,t) = \lambda^{\gamma-1} \theta( \lambda x,\lambda^\gamma t)$ is a 
$\TT_\lambda=  [0,1/\lambda]^2$-periodic solution of \eqref{eq:SQG:gamma} with initial datum $\theta_{0,\lambda}(x) = \lambda^{\gamma-1} \theta(\lambda x)$.
If $\gamma = 1$ the $L^\infty$-norm is thus scaling invariant, i.e. $\| \theta_{0,\lambda}\|_{L^\infty} = \|\theta_0\|_{L^\infty}$ for any $\lambda >0$, and this case is referred to
as \emph{critical}. In the \emph{supercritical} case $\gamma\in (0,1)$, examples of scale-invariant norms for the initial datum include the H\"older space
$C^{1-\gamma}$ and the Sobolev space $\dot{H}^{2-\gamma}$ (both compactly embed in $L^\infty$). However, we are not aware of any global in time \emph{a priori} estimate  available for such strong norms, which makes studying the regularity of solutions with {arbitrarily large initial datum} a challenging  problem.

While for the critical case $\gamma=1$ the question of global regularity of \eqref{eq:SQG:gamma} with arbitrarily large datum has been settled~\cites{CV10a,CV12,CTV13,KN09,KNV07} (see also~\cite{DKSV14} for the logarithmically supercritical case), the corresponding result for the 
supercritical equation $\gamma\in (0,1)$ remains 
open. The global existence is only known for data that are {\em small} in a suitable scaling invariant space $X$. This program started with~\cite{CCW00}. Roughly speaking, the a priori estimate that is usually proven for such results is of the type
\[
\frac{\d}{\d t} \| \theta\|_{X}^2 + \| \Lambda^{\gamma/2} \theta\|_{X}^2 \leq C \| \theta\|_{X} \| \Lambda^{\gamma/2} \theta\|_{X}^2
\]
where $C$ is a sufficiently large constant that depends e.g.~on $\| \RR^\perp \|_{L^p \to L^p}$, and in particular $C$ is larger than $1$ (or a constant {\em independent} of $\gamma$). Thus, if $\|\theta_0\|_{X} \leq 1/C$ then $\|\theta(t)\|_{X} \leq 1/C$ for all $t\geq 0$ and the global existence follows from the local existence theorem with data in $X$.
Specifically, \cites{CC04, Miu06, Ju07,Yu08,Dong10} deal with the Sobolev space
setting, showing that small initial data (with respect to the viscosity coefficient, here normalized to 1) in $X = H^{2-\gamma}$ lead to the global existence
of solutions. Similar results were obtained in \cites{CL03, Wu05, CMZ07, HK07,Wu07} for the critical Besov spaces $X = B^{1-\gamma + 2/p}_{p,1}$.
However, none of these results seems to yield the global well-posedness of solutions for initial datum of size $\gg 1$.

In this paper we consider the scaling invariant norm $\| \cdot \|_X = \|\cdot\|_{L^2}^{\gamma/2} \|\cdot \|_{\dot{H}^{2}}^{1-\gamma/2}$. For $\gamma \in (0,1]$ we define
\begin{align}
R_\gamma = \sup \{ R > 0 \colon 
&\mbox{ for any } \theta_0\in H^2  \mbox{ with } \| \theta_0\|_{L^2}^{\gamma/2}\|\theta_0\|_{\dot{H}^2}^{1-\gamma/2} \leq R, \mbox{  the unique smooth} \notag\\ 
&\mbox{ solution of }\eqref{eq:SQG:gamma}\mbox{ with initial data } \theta_0 \mbox{ does not blow up in finite time} \}.
\label{eq:R:gamma}
\end{align}
From the small data results for $\gamma \in (0,1)$ we know that $R_\gamma >0$, while from 
the global regularity results in the critical case we have that $R_1 = \infty$. {\em The question 
we address in this paper is  whether or not $R_\gamma \to \infty$ as $\gamma \to 1$.}  We 
answer this question in the affirmative and show that:
\begin{theorem}\label{thm:superGWP}
Let $\theta_0\in H^2$ with $\| \theta_0\|_{L^2}^{\gamma/2}\|\theta_0\|_{\dot{H}^2}^{1-\gamma/2} \leq R$. There exists $\gamma_1=\gamma_1(R)\in (0,1)$ such that for every
$\gamma\in [\gamma_1,1)$ the initial value problem for the supercritical SQG equation \eqref{eq:SQG:gamma} with initial datum $\theta_0$ has
a unique global in time solution $\theta:\TT\times[0,\infty)\to \R$, with
\begin{align*}
\theta\in L_{loc}^\infty(0,\infty;H^2)\cap L_{loc}^2(0,\infty; H^{2+\gamma/2})
\end{align*}
and is therefore $\theta$ is a classical solution. 
\end{theorem}
\begin{corollary}
For $\gamma \in (0,1)$ let $R_\gamma$ be as defined in \eqref{eq:R:gamma}. Then $R_\gamma \to\infty$ as $\gamma \to 1$.
\end{corollary}
The above result expresses a continuity of the solution map of \eqref{eq:SQG:gamma} 
with respect to the parameter $\gamma$, as $\gamma \to 1$. The proof of 
Theorem~\ref{thm:superGWP} proceeds as follows. Given any large datum 
$\theta_0 \in H^{2}$, there exists a unique local in time solution on $[0,T_1)$, for 
some $T_1>0$ that depends on $\|\theta_0\|_{H^{2}}$ (cf.~e.g.~\cite{Ju07}). We emphasize 
that $T_1$ is not known a priori to depend solely on $R$, or any other scaling-critical norm 
of $\theta_0$. Moreover, on $[0,T_1)$ the solution becomes smooth~\cites{Dong10,DDL09}. 
On the other hand we know that there exists an eventual regularization time $T_*$ such that if 
the solution does not blow up on $[0,T_*]$, then it cannot blow up on $[T_*, \infty)$ either 
(cf.~\cites{Sil10a,Dab11,Kis11}). It remains to show that $T_1 > T_*$ for $\gamma$ sufficiently 
close to $1$, which depends on the data only through the bound $R$. The difficulty in executing 
this plan lies in keeping track of the precise dependence of all estimates in terms of $\gamma$, 
as $\gamma \to 1$, and on $\theta_0$.  For this purpose we need to have an accurate estimate on 
how the eventual regularization time $T_*$ depends on the initial datum and on the power of the 
fractional Laplacian. 
We give a new proof of eventual regularity, that is based on the method of~\cite{CTV13}, and in 
particular on nonlinear lower bounds for the fractional Laplacian established in~\cite{CV12}. 
Moreover we obtain a quantitative upper bound for $T_*$ that depends explicitly on $\gamma$ 
and $\|\theta_0\|_{L^\infty}$. Our result is:

\begin{theorem}\label{thm:eventual}
Fix $\gamma\in [\gamma_0,1)$, $\theta_0 \in L^\infty$, and let $\alpha\in(1-\gamma,1)$ be arbitrary. 
Let $C>0$ be a positive sufficiently large universal constant, and define the time
\begin{align}\label{eq:time}
T_\star=C\alpha^{\frac{\gamma(2-\gamma)}{1-\gamma}} \|\theta_0\|_{L^\infty}^{\frac{\gamma}{1-\gamma}}.
\end{align}
If $\theta \in C^\infty(\TT \times [0,T_\star])$ is a smooth solution of \eqref{eq:SQG:gamma}, then $\theta \in C^\infty(\TT \times [0,\infty))$, and the bound
\begin{equation}\label{eq:alphabdd}
[\theta(t)]_{C^\alpha}\leq C \alpha^{-\frac{\alpha}{1-\gamma}} \|\theta_0\|_{L^\infty}^{- \frac{\gamma+\alpha-1}{1-\gamma} }
\end{equation}
holds for all $t\geq T_\star$.
\end{theorem}

The fact that $T_\star$ depends only on $\|\theta_0\|_{L^\infty}$, which for $\gamma <1$ is below the critical regularity level, compensates for the fact that the local existence time depends on norms above the critical regularity level. This observation is key for completing the proof of Theorem~\ref{thm:superGWP}.
 
We note that Theorem~\ref{thm:eventual} holds in fact for any  global weak solution  obtained from regularizations which respect the maximum principle, e.g. viscosity solutions obtained as limits when $\epsilon \to 0$ from a hyper-viscous $-\epsilon \Delta$ regularization.  Note moreover that by passing $\gamma\to 1$, upon choosing $\alpha$ small enough depending on $\|\theta_0\|_{L^\infty}$, 
Theorem~\ref{thm:eventual} shows that  the critical SQG equation regularizes instantaneously from $L^\infty$ to $C^\alpha$. 

The eventual regularity of weak solutions to supercritical SQG has been previously established in \cite{Sil10a} for $\gamma$ sufficiently close to $1$, using the techniques of \cite{CV10a} (see also~\cite{CCS10}), in \cite{Dab11} for the full range $\gamma\in (0,1)$ by means of the methods  devised in \cite{KN09}, and finally in \cite{Kis11} for all $\gamma\in (0,1)$ through a modification of the modulus of continuity approach that has been successfully employed in \cites{KNV07,KNS08,Kis11}. The simple proof given in this paper is based on the arguments in~\cites{CTV13,CV12}, cf.~Section~\ref{sub:proof1} below.

\subsection*{Organization of the paper} In Section~\ref{sec:prelim} we recall the definitions of the operators and spaces used in the paper. The proof of Theorem~\ref{thm:eventual} is given in Section~\ref{sec:event}. In Section~\ref{sec:local} we bound from below the local existence time. Lastly, the proof of Theorem~\ref{thm:superGWP} is given in Section~\ref{sec:theproof}.


\section{Preliminaries}~\label{sec:prelim}

\subsection*{Fractional Laplacian}
The fractional Laplacian $\Lambda^\sigma$,  can be defined for $\sigma\in (0,2)$ as the Fourier
multiplier with symbol $|k|^\sigma$, or in physical variables by 
\begin{align*}
\Lambda^\sigma \varphi(x)= c_\sigma \sum_{k \in \ZZ^2}  \int_{\TT}   \frac{\varphi(x) - \varphi(x+y)}{|y- 2\pi k|^{2+\alpha}} \d y= c_\sigma \,\mathrm{P.V.}\int_{\R^2} \frac{\varphi(x)-\varphi(x+y)}{|y|^{2+\sigma}}\d y, 
\end{align*} 
valid for $\varphi \in C^{\sigma + \eps}(\TT)$ for some $\eps>0$. In the above identity and throughout the paper we abuse notation and still denote by $\varphi$ the periodic extension of $\varphi$ to the whole space. The precise form of the constant $c_\sigma>0$ is not important
for our purposes and for $\sigma \in [\gamma_0,1]$ we have $c_\sigma$ bounded from above and below in terms of universal constants and $\gamma_0$. This is because we do not pass to the limits $\sigma \to 0^+$ or $\sigma \to 2^-$.

\subsection*{Velocity constitutive law}
The velocity vector field $\u$ in \eqref{eq:SQG:gamma} is divergence-free and determined by $\theta$ through the relation
$\u=\RR^\perp \theta = \nabla^\perp\Lambda^{-1}\theta=(-\de_{x_2}\Lambda^{-1}\theta,\de_{x_1}\Lambda^{-1}\theta) =(-\RR_2\theta,\RR_1 \theta)$,
where 
\begin{align*}
\RR_j \varphi(x)&= \frac{1}{2\pi} \mathrm{P.V.} \int_{\TT} \frac{y_j}{|y|^3} \varphi(x+y) \d y + \sum_{k \in \ZZ^2_*} \int_{\TT} \left( \frac{y_j + 2 \pi k_j }{|y + 2\pi k |^3}  - \frac{2 \pi k_j }{|2\pi k |^3} \right) \varphi(x+y) \d y \\
&=\frac{1}{2\pi}\,\mathrm{P.V.}\int_{\R^2}\frac{y_j}{|y|^3}\varphi(x+y)\d y.
\end{align*}
In the last line the principal value is taken both as $|y| \to 0$ and $|y|\to \infty$.

\subsection*{Spaces}
Throughout the article we consider mean-zero solutions to \eqref{eq:SQG:gamma}, so that we will not
make a distinction between homogenous and inhomogenous spaces. For $p\in[1,\infty]$ the Lebesgue 
norm is denoted by $\|\cdot \|_{L^p}$, for $s \in \R$ the Sobolev norms are denoted 
by $\| \cdot\|_{\dot{H}^s} = \| \Lambda^s \cdot \|_{L^2}$,  and for $\alpha\in(0,1)$ the usual H\"older norm
is given by
$\|\varphi\|_{C^\alpha}=\|\varphi\|_{L^\infty}+[\varphi]_{C^\alpha}$,
where $[\varphi]_{C^\alpha}=\sup_{x\neq y \in\TT} |\varphi(x)-\varphi(y)| |x-y|^{-\alpha}$.

\subsection*{Notation}
Throughout the paper, $C$ will denote a \emph{generic} positive constant,  
whose value may change even in the same line of a certain equation. In the 
same spirit, $c,c_0,c_1,\ldots$ will denote fixed constants appearing in the course of proofs 
or estimates, which have to be referred to specifically. In an essential way, throughout this paper {\em the dependence of various constants on the parameters $\gamma \in [\gamma_0,1)$ and $\alpha \in (1-\gamma , 1)$ will be emphasized only
when $\gamma \to 1$ or $\alpha\to 0$}.


\section{Eventual regularity for supercritical SQG}\label{sec:event}

In this section we give the proof of Theorem~\ref{thm:eventual}. Since $\alpha > 1-\gamma$, in view of the conditional regularity results of \cites{CW08,DP09} (which are known to be sharp in the case of linear drift-diffusion equations~\cite{SVZ13})
once $\theta \in L^\infty(0,T; C^{\alpha})$ we automatically have $\theta \in L^\infty(0,T; C^1)$ 
and thus the solution is classical on $[0,T]$. Further regularity follows from standard parabolic bootstrap 
arguments. Thus, our proof of Theorem~\ref{thm:eventual}  only consists in proving the bound
\eqref{eq:alphabdd}. We start with a number of preliminary results and the proof is postponed to 
Section~\ref{sub:proof1} below.

\subsection{Finite differences and H\"older norms}
In order to estimate $C^\alpha$-seminorms it is natural to consider  the finite difference
\begin{align*}
\delta_h\theta(x,t)=\theta(x+h,t)-\theta(x,t),
\end{align*}
which is periodic in both $x$ and $h$, where  $x,h \in \TT$.  As in \cites{CV12,CTV13}, it follows that
\begin{equation}\label{eq:findiff}
L (\delta_h\theta)^2+ D_\gamma[\delta_h\theta]=0,
\end{equation}
where $L$ denotes the differential operator
\begin{equation}
\label{eq:L:def}
L=\de_t+\u\cdot \nabla_x+(\delta_h\u)\cdot \nabla_h+ \Lambda^\gamma.
\end{equation}
and 
\begin{equation}
\label{eq:D:gamma:def}
D_\gamma[\varphi](x)= c_\gamma \int_{\R^2} \frac{\big[\varphi(x)-\varphi(x+y)\big]^2}{|y|^{2+\gamma}}\d y.
\end{equation}
Here we have used that for $\gamma\in (0,2)$ and $\varphi\in C^\infty(\TT)$, cf. \cite{CC04} we have that
\begin{align*}
2\varphi(x) \Lambda^\gamma \varphi(x)=\Lambda^\gamma \big(\varphi(x)^2\big)+D_\gamma[\varphi](x),
\end{align*}
pointwise for $x\in\TT$. 
Let $\xi:[0,\infty)\to[0,\infty)$ be a bounded decreasing differentiable function to be determined later. For 
\begin{equation}\label{eq:alpha}
\alpha \in (1-\gamma, 1),
\end{equation}
we want to study the evolution of the quantity $v(x,t;h)$ defined by
\begin{equation}\label{eq:v}
v(x,t;h) =\frac{\delta_h\theta(x,t) }{(\xi(t)^2+|h|^2)^{\alpha/2}}.
\end{equation}
The main point is that when $\xi(t)=0$ the quantity
$$
\|v(t)\|_{L^\infty_{x,h}}=\esup_{x,h\in\TT} |v(x,t;h)|
$$
is equivalent to the H\"older seminorm $[\theta(t)]_{C^\alpha}$, while for $\xi(t)>0$  we have that  $\|v(t)\|_{L^\infty_{x,h}} \leq 2 \|\theta(t)\|_{L^\infty} \xi(t)^{-\alpha}$.
From \eqref{eq:findiff} we learn that
\begin{align}
L v^2+\frac{1}{(\xi^2+|h|^2)^\alpha}D_\gamma[\delta_h\theta] 
&=2\alpha |\dot\xi|\frac{\xi}{\xi^2+|h|^2}v^2 -2\alpha \frac{h}{\xi^2+|h|^2}\cdot \delta_h\u \, v^2 \notag \\
&\leq 2\alpha |\dot\xi|\frac{\xi}{\xi^2+|h|^2}v^2 +2\alpha \frac{|h|}{(\xi^2+|h|^2)}|\delta_h\u|v^2\label{eq:ineq1}
\end{align}
where $\delta_h\u= \RR^\perp \delta_h\theta$. 
The goal of this section is to provide
a suitable uniform bound on $v$ by deriving a number of estimates on the right-hand side of \eqref{eq:ineq1}. 

\subsection{Nonlinear lower bounds}
We begin by deriving a lower bound on $D_\gamma[\delta_h\theta]$, which we state in the 
following lemma.

\begin{lemma}\label{lem:nonlinbdd}
Let $\gamma\in [\gamma_0,1)$ and $\alpha\in(1-\gamma,1)$. Then there exists a positive constant
$c_0=c_0(\gamma_0)$ such that
\begin{equation}\label{eq:N1}
D_\gamma[\delta_h\theta](x)
\geq \frac{1}{c_0|h|^\gamma} \left[\frac{|v(x;h)|}{\|v\|_{L^\infty_{x,h}}}\right]^{\frac{\gamma}{1-\alpha}}|\delta_h\theta(x)|^2,
\end{equation}
holds for any $x,h\in \TT$. Consequently,
\begin{equation}\label{eq:N2}
\frac{1}{(\xi^2+|h|^2)^\alpha}D_\gamma[\delta_h\theta](x) 
\geq \frac{1}{c_0|h|^\gamma} \left[\frac{|v(x;h)|}{\|v\|_{L^\infty_{x,h}}}\right]^{\frac{\gamma}{1-\alpha}} (v(x;h))^2
\end{equation}
holds pointwise.
\end{lemma}

\begin{proof}[Proof of Lemma~\ref{lem:nonlinbdd}]
In what follows, we will neglect the dependence on $t$ of the functions involved. It is understood that all the estimates
below are valid pointwise in $t\geq 0$. Also, it is enough to prove \eqref{eq:N1}, as \eqref{eq:N2} follows directly
from the definition of $v(x;h)$.

Let $\chi$ be a smooth radially non-increasing cutoff function that vanishes on $|x|\leq 1$ and is identically 1 for $|x|\geq 2$ and
such that $|\chi'|\leq 2$.
For $R\geq 4|h|$, we have
\begin{align}\label{eq:longcalc}
D_\gamma[\delta_h\theta](x)
&\geq c_\gamma \int_{\R^2} \frac{\big[\delta_h\theta(x)-\delta_h\theta(x+y)\big]^2}{|y|^{2+\gamma}}\chi(|y|/R)\d y \notag \\
&\geq c_\gamma |\delta_h\theta(x)|^2\int_{\R^2} \frac{\chi(|y|/R)}{|y|^{2+\gamma}}\d y 
-2c_\gamma |\delta_h\theta(x)| \left|  \int_{\R^2} \frac{\delta_h\theta(x+y)}{|y|^{2+\gamma}}\chi(|y|/R)\d y        \right| \notag \\
&\geq  c_\gamma \frac{|\delta_h\theta(x)|^2}{R^\gamma} 
-2c_\gamma |\delta_h\theta(x)| \left|  \int_{\R^2} \big[\theta(x+y)-\theta(x)\big]\delta_{-h} \frac{\chi(|y|/R)}{|y|^{2+\gamma}}\d y \right| \notag \\
&\geq  c_\gamma \frac{|\delta_h\theta(x)|^2}{R^\gamma} 
-c_1c_\gamma |\delta_h\theta(x)|  \, |h|   
\int_{|y|\geq R} \frac{|\delta_y\theta(x)|}{(\xi^2+|y|^2)^{\alpha/2}} \frac{(\xi^2+|y|^2)^{\alpha/2}}{|y|^{3+\gamma}}\d y \notag \\
&\geq  c_\gamma \frac{|\delta_h\theta(x)|^2}{R^\gamma} -c_1c_\gamma |\delta_h\theta(x)|  \, |h|  \, \|v\|_{L^\infty_{x,h}}
\int_{R}^\infty  \frac{(\xi^2+\rho^2)^{\alpha/2}}{\rho^{2+\gamma}}\d \rho,
\end{align}for some constant $c_1\geq 1$. First, notice that
\begin{equation}\label{eq:bdd1}
\int_{R}^\infty  \frac{(\xi^2+\rho^2)^{\alpha/2}}{\rho^{2+\gamma}}\d \rho \leq c_2\big(\xi^\alpha R^{-1-\gamma}+R^{-1-\gamma+\alpha}\big)
=\frac{c_2}{R^\gamma}\left(\frac{\xi^\alpha}{R}+\frac{1}{R^{1-\alpha}}\right),
\end{equation}
for some $c_2\geq1$.
We now choose $R>0$ as
\begin{equation}\label{eq:Rchoice}
R=\left[\frac{4c_1c_2  (\xi^2+|h|^2)^{\alpha/2} \|v\|_{L^\infty_{x,h}} }{ |\delta_h\theta(x)|}\right]^{\frac{1}{1-\alpha}}|h|=\left[4c_1c_2\frac{ \|v\|_{L^\infty_{x,h}} }{ |v(x;h)|}\right]^{\frac{1}{1-\alpha}}|h|.
\end{equation}
Since $c_1c_2\geq 1$ and $v(x;h)\leq  \|v\|_{L^\infty_{x,h}}$, it is apparent from \eqref{eq:Rchoice} that 
\[
R\geq 4^{\frac{1}{1-\alpha}}|h|\geq 4|h|,
\]
where the last inequality follows from the assumption $\alpha<1$. By using
\eqref{eq:Rchoice} and the trivial estimates 
\begin{align*}
\left[ \frac{|v(x;h)|}{\|v\|_{L^\infty_{x,h}}}\right]^{\frac{1}{1-\alpha}}\leq \frac{|v(x;h)|}{\|v\|_{L^\infty_{x,h}}}
\qquad
\mbox{and}
\qquad
\left[ \frac{1}{4c_1c_2}\right]^{\frac{1}{1-\alpha}}\leq \frac{1}{4c_1c_2}.
\end{align*}
we rewrite the bound \eqref{eq:bdd1} as
\begin{align}\label{eq:bdd2}
\int_{R}^\infty  \frac{(\xi^2+\rho^2)^{\alpha/2}}{\rho^{2+\gamma}}\d \rho 
&\leq \frac{c_2}{R^\gamma}\left( 
\frac{\xi^\alpha}{(4c_1c_2)^{\frac{1}{1-\alpha}} |h|}\left[ \frac{|v(x;h)|}{\|v\|_{L^\infty_{x,h}}}\right]^{\frac{1}{1-\alpha}}    
+\frac{1}{4c_1c_2 |h|^{1-\alpha}}\frac{|v(x;h)|}{\|v\|_{L^\infty_{x,h}}} \right) \notag \\
&\leq  \frac{1}{4c_1 R^\gamma}\left( 
\frac{\xi^\alpha}{|h|}    
+\frac{|h|^\alpha}{|h|}\right)\frac{|v(x;h)|}{\|v\|_{L^\infty_{x,h}}} \notag \\
&\leq \frac{1}{2c_1 R^\gamma}\frac{(\xi^2+|h|^2)^{\alpha/2}}{|h|}\frac{|v(x;h)|}{\|v\|_{L^\infty_{x,h}}} 
=\frac{1}{2c_1 R^\gamma}\frac{1}{|h|}\frac{|\delta_h\theta(x)|}{\|v\|_{L^\infty_{x,h}}}.
\end{align}
Hence, combining the estimate in \eqref{eq:longcalc} with the above \eqref{eq:bdd2}, we arrive at
\begin{align*}
D_\gamma[\delta_h\theta](x)\geq \frac{c_\gamma}{2R^\gamma}|\delta_h\theta(x)|^2.
\end{align*}
Estimate \eqref{eq:N1} now follows immediately from the expression of $R$ in \eqref{eq:Rchoice}.
\end{proof}

\subsection{The differential equation for $\xi$}
We now establish the differential equation that $\xi$ has to satisfy to control the first
term of the right-hand side of \eqref{eq:ineq1} with a fraction of the nonlinear lower bound \eqref{eq:N2}.
\begin{lemma}\label{lem:ODE}
Let $\gamma\in [\gamma_0,1)$ and $\alpha\in(1-\gamma,1)$. There exists a positive constant $c_\star=c_\star(\gamma_0)$
such that if 
\begin{equation}\label{eq:ODE}
\dot\xi =-\frac{c_\star}{\alpha} \xi^{1-\gamma},
\end{equation}
then the estimate
\begin{equation}\label{eq:est1}
2\alpha |\dot\xi|\frac{\xi}{\xi^2+|h|^2}v^2\leq \frac{1}{8c_0 |h|^\gamma} v^2,
\end{equation}
holds pointwise for $x,h \in \TT$, where $c_0$ is the same constant appearing in \eqref{eq:N2}.
\end{lemma}

\begin{proof}[Proof of Lemma~\ref{lem:ODE}]
The constant $c_\star>0$ will be determined at the end of the proof. If $\xi$ obeys \eqref{eq:ODE},
then
\begin{align*}
2\alpha |\dot\xi|\frac{\xi}{\xi^2+|h|^2}v^2
=2c_\star\frac{\xi^{2-\gamma}}{\xi^2+|h|^2}v^2
\leq 2c_\star\frac{(\xi^2+|h|^2)^{1-\gamma/2}}{\xi^2+|h|^2}v^2
=\frac{2c_\star}{(\xi^2+|h|^2)^{\gamma/2}}v^2
\leq \frac{2c_\star}{|h|^\gamma}v^2.
\end{align*}
Thus, the claim follows by setting 
$
c_\star= 1/(16c_0).
$
\end{proof}
We observe here that the initial condition for the differential equation \eqref{eq:ODE} is yet to be determined. Its
value will be computed in the subsequent paragraph. Given $\xi(0)=\xi_0>0$, a solution to  \eqref{eq:ODE} is given 
by
\[
\xi(t)=\begin{cases}
\displaystyle \left[\xi_0^\gamma-\frac{\gamma c_\star}{\alpha} t\right]^{1/\gamma}, \quad &\text{if } t\in [0,T_\star],\\ \\
0,\quad &\text{if } t \in  (T_\star,\infty),
\end{cases}
\]
where 
\begin{equation}\label{eq:regtime}
T_\star=\frac{\alpha}{\gamma c_\star} \xi_0^\gamma.
\end{equation}

\subsection{Estimates on the nonlinear term}
Following the ideas of \cites{CTV13, CV12}, we now consider the second term
in the right-hand side of  \eqref{eq:ineq1}, in order to derive a suitable upper bound in terms
of the dissipation. We begin with a result involving solely the term $\delta_h\u$. 

\begin{lemma}\label{lem:rieszbdd}
Let $\rho\geq 4|h|$ be arbitrarily fixed. Then
\begin{equation}\label{eq:riesz1}
|\delta_h\u(x)|\leq 
C\left[ \rho^{\gamma/2} \big(D_\gamma[\delta_h\theta](x)\big)^{1/2}+\frac{|h|\|v\|_{L^\infty_{x,h}} \xi^\alpha}{\rho}
+\frac{|h|\|v\|_{L^\infty_{x,h}} }{\rho^{1-\alpha}}\right],
\end{equation}
holds pointwise in $x,h\in \TT$.
\end{lemma}
\begin{proof}[Proof of Lemma~\ref{lem:rieszbdd}]
Let us fix $\rho\geq 4|h|$. As before, let $\chi$ be a smooth radially non-increasing cutoff function 
that vanishes on $|x|\leq 1$ and is identically 1 for $|x|\geq 2$ and such that $|\chi'|\leq 2$.
We split the vector $\delta_h\u$ in an inner and an outer part
\begin{align*}
\delta_h\u(x)
=\frac{1}{2\pi} \mathrm{P.V.}\int_{\R^2} \frac{y^\perp}{|y|^3}\big[\delta_h\theta(x+y)-\delta_h\theta(x)\big]\d y
=\delta_h\u_{in}(x)+\delta_h\u_{out}(x),
\end{align*}
by using that the kernel of $\RR^\perp$ has zero average on the unit sphere, where
\begin{align*}
\delta_h\u_{in}(x)=\frac{1}{2\pi}\mathrm{P.V.}\int_{\R^2} \frac{y^\perp}{|y|^3}\big[1-\chi(|y|/\rho)\big] \big[\delta_h\theta(x+y)-\delta_h\theta(x)\big]\d y,
\end{align*}
and
\begin{align*}
\delta_h\u_{out}(x)&=\frac{1}{2\pi}\mathrm{P.V.}\int_{\R^2} \frac{y^\perp}{|y|^3} \chi(|y|/\rho)\big[\delta_h\theta(x+y)-\delta_h\theta(x)\big]\d y\\
&=\frac{1}{2\pi}\mathrm{P.V.}\int_{\R^2} \delta_{-h}\left[\frac{y^\perp}{|y|^3}\chi(|y|/\rho) \right]\big[\theta(x+y)-\theta(x)\big]\d y.
\end{align*}
For the inner piece, in light of the Cauchy-Schwartz inequality, we obtain
\begin{align}
|\delta_h\u_{in}(x)|&\leq\frac{1}{2\pi}\int_{|y|\leq \rho} \frac{1}{|y|^2}|\delta_h\theta(x+y)-\delta_h\theta(x)|\d y
\notag \\
&\leq \frac{1}{2\pi} \left[ \int_{|y|\leq \rho} \frac{1}{|y|^{2-\gamma}}  \right]^{1/2}\left[  \int_{\R^2}\frac{(\delta_h\theta(x+y)-\delta_h\theta(x))^2}{|y|^{2+\gamma}} \d y\right]^{1/2}
\notag \\
&\leq C\rho^{\gamma/2} \big(D_\gamma[\delta_h\theta](x)\big)^{1/2}.
\label{eq:pfriesz1}
\end{align}
Regarding the outer part, the mean value theorem entails
\begin{align}\label{eq:pfriesz2}
|\delta_h\u_{out}(x)|&\leq C|h|\int_{|y|\geq \rho/2} \frac{(\xi^2+|y|^2)^{\alpha/2}}{|y|^3}\frac{|\theta(x+y)-\theta(x)|}{(\xi^2+|y|^2)^{\alpha/2}}\d y \notag \\
&\leq  C|h|\|v\|_{L^\infty_{x,h}}\int_{|y|\geq \rho/2} \frac{(\xi^2+|y|^2)^{\alpha/2}}{|y|^3}\d y \notag \\
&\leq C|h|\|v\|_{L^\infty_{x,h}}\left[\frac{\xi^\alpha}{\rho}+\frac{1}{\rho^{1-\alpha}}\right].
\end{align}
The conclusion follows by combining \eqref{eq:pfriesz1} and \eqref{eq:pfriesz2}. 
\end{proof}
Using Lemma \ref{lem:rieszbdd} we are able to properly compare the nonlinear term
in \eqref{eq:ineq1} with the lower bound on the dissipation term given by \eqref{eq:N2}.

\begin{lemma}\label{lem:ubdd}
Let $\gamma\in [\gamma_0,1)$, $\alpha\in(1-\gamma,1)$, and assume that
\begin{equation}\label{eq:vbdd0}
\|v\|_{L^\infty_{x,h}}\leq M:=\frac{4\|\theta_0\|_{L^\infty}}{\xi_0^\alpha}.
\end{equation}
There exists a constant $c_1=c_1(\gamma_0)\geq 1$ such that if 
\begin{equation}\label{eq:xi0}
\xi_0=   (c_1\alpha\|\theta_0\|_{L^\infty})^{1/(1-\gamma)},
\end{equation}
then the estimate
\begin{equation}\label{eq:est2}
2\alpha  \frac{|h|}{\xi^2+|h|^2}|\delta_h\u|v^2\leq \frac{1}{2(\xi^2+|h|^2)^\alpha}D_\gamma[\delta_h\theta] + \frac{1}{8c_0 |h|^\gamma} v^2,
\end{equation}
holds pointwise for every 
\[
x,h \in \TT \qquad \mbox{with} \qquad |h|\leq \xi_0,
\]
where $c_0$ is the constant appearing in \eqref{eq:N2}.
\end{lemma}

\begin{proof}[Proof of Lemma~\ref{lem:ubdd}]
The nonlinear term in \eqref{eq:ineq1} can be estimated using \eqref{eq:riesz1}  as
\begin{align}
2\alpha  \frac{|h|}{\xi^2+|h|^2}|\delta_h\u|v^2
&\leq  C\alpha\frac{|h|}{\xi^2+|h|^2}C\left[ \rho^{\gamma/2} \big(D_\gamma[\delta_h\theta](x)\big)^{1/2}+\frac{|h|\|v\|_{L^\infty_{x,h}} \xi^\alpha}{\rho}
+\frac{|h|\|v\|_{L^\infty_{x,h}} }{\rho^{1-\alpha}}\right]v^2\notag\\
&\leq \frac{1}{2(\xi^2+|h|^2)^\alpha}D_\gamma[\delta_h\theta] \notag\\
&\qquad+C\alpha\frac{|h|^2}{\xi^2+|h|^2}\left[\frac{\alpha v^2}{(\xi^2+|h|^2)^{1-\alpha}}\rho^\gamma+\frac{\|v\|_{L^\infty_{x,h}} \xi^\alpha}{\rho}
+\frac{\|v\|_{L^\infty_{x,h}} }{\rho^{1-\alpha}}\right]v^2.
\label{eq:nonlin1}
\end{align}
We first focus on the last term in the above inequality. We choose $\rho$ as
\begin{align*}
\rho=4(\xi^2+|h|^2)^{1/2}.
\end{align*}
Obviously $\rho \geq 4|h|$. Now, using that $\alpha+\gamma-1>0$ and \eqref{eq:vbdd0}, we find that
\begin{align}
\frac{\alpha v^2}{(\xi^2+|h|^2)^{1-\alpha}}\rho^\gamma
\leq C\alpha \frac{ M^2}{(\xi^2+|h|^2)^{1-\alpha}}(\xi^2+|h|^2)^{\gamma/2}  
\leq C\alpha\frac{ \|\theta_0\|^2_{L^\infty}}{\xi_0^{2(1-\gamma)}}\frac{1}{|h|^\gamma}.
\label{eq:comp1}
\end{align}
In the last inequality above, we have recalled the definition of $M$ in \eqref{eq:vbdd0} 
and used the bound
\[
\frac{(\xi^2+|h|^2)^{\gamma/2}}{\xi_0^{2\alpha}(\xi^2+|h|^2)^{1-\alpha}} \leq \frac{ C }{\xi_0^{2(1-\gamma)}|h|^\gamma}
\]
which holds since $\alpha +\gamma > 1$, we have chosen $|h| \leq   \xi_0$, and by definition we have $\xi(t) \leq \xi_0$. This is in fact the only place in the proof where the restriction $|h| \leq \xi_0$ is  used. 
For the other two terms in \eqref{eq:nonlin1}, we have
\begin{align}
\frac{\|v\|_{L^\infty_{x,h}} \xi^\alpha}{\rho}+\frac{\|v\|_{L^\infty_{x,h}} }{\rho^{1-\alpha}}
&\leq CM \left[\frac{\xi^\alpha}{(\xi^2+|h|^2)^{1/2}}+\frac{1}{(\xi^2+|h|^2)^{(1-\alpha)/2}}\right] \notag\\
&\leq C\frac{\|\theta_0\|_{L^\infty}}{\xi_0^\alpha} \frac{(\xi^2+|h|^2)^{(\alpha+\gamma-1)/2}}{|h|^{\gamma}}
\leq C\frac{\|\theta_0\|_{L^\infty}}{\xi_0^{1-\gamma}} \frac{1}{|h|^{\gamma}}.\label{eq:comp2}
\end{align}
In light of \eqref{eq:comp1} and \eqref{eq:comp2}, we can rewrite \eqref{eq:nonlin1} as
\begin{align}
2\alpha  \frac{|h|}{\xi^2+|h|^2}|\delta_h\u|v^2 
&\leq \frac{1}{2(\xi^2+|h|^2)^\alpha}D_\gamma[\delta_h\theta] +
C\alpha\frac{|h|^2}{\xi^2+|h|^2}\left[\alpha \frac{ \|\theta_0\|^2_{L^\infty}}{\xi_0^{2(1-\gamma)}} + \frac{\|\theta_0\|_{L^\infty}}{\xi_0^{1-\gamma}} \right] \frac{1}{|h|^{\gamma}} v^2\notag \\
&\leq  \frac{1}{2(\xi^2+|h|^2)^\alpha}D_\gamma[\delta_h\theta] +
C\alpha\left[\alpha \frac{ \|\theta_0\|^2_{L^\infty}}{\xi_0^{2(1-\gamma)}} + \frac{\|\theta_0\|_{L^\infty}}{\xi_0^{1-\gamma}} \right] \frac{1}{|h|^{\gamma}} v^2. \label{eq:nonlin2}
\end{align}
Henceforth, we require  $\xi_0$ big enough so that 
\begin{align*}
C\alpha\left[\alpha \frac{ \|\theta_0\|^2_{L^\infty}}{\xi_0^{2(1-\gamma)}} + \frac{\|\theta_0\|_{L^\infty}}{\xi_0^{1-\gamma}} \right]\leq \frac{1}{8c_0},
\end{align*}
where $c_0>0$ is the constant appearing in \eqref{eq:N2}.  
The above requirement is fulfilled in particular if we impose
\begin{align*}
\frac{\|\theta_0\|_{L^\infty}}{\xi_0^{1-\gamma}}\leq \frac{1}{16Cc_0\alpha},
\end{align*}
namely the lower bound
\[
\xi_0^{1-\gamma} \geq 16C c_0 \alpha \|\theta_0\|_{L^\infty},
\]
which concludes the proof of the lemma.
\end{proof}

\subsection{Proof of Theorem \ref{thm:eventual}}\label{sub:proof1}
We are now ready to prove Theorem \ref{thm:eventual}. Define $\xi_0$ as in \eqref{eq:xi0}.
From the definition of $v$ in \eqref{eq:v}, it is immediate to see that
\begin{align*}
\|v(0)\|_{L^\infty_{x,h}}\leq \frac{2\|\theta_0\|_{L^\infty}}{\xi_0^\alpha} =: \frac{M}{2}
\end{align*}
Define 
\begin{align*}
t_0=\sup\{t\geq 0: \|v(\tau)\|_{L^\infty_{x,h}} < M, \ \forall \tau \in[0,t]\}.
\end{align*}
In other words, $t_0$ is the first time for which  $\|v(t)\|_{L^\infty_{x,h}}$ reaches the value $M$. We claim that $t_0=\infty$. 
Since $t\mapsto \|v(t)\|_{L^\infty_{x,h}}$ is a continuous function, we clearly have that $t_0>0$. 

Due to the smoothness of $v$ in $x$ and $h$, and the periodicity of $\delta_h \theta(x)$ 
in both $x$ and $h$, there exist $(\bar x , \bar h) \in \TT \times \TT$ with 
$|v(\bar x, t_0; \bar h)| = \|v(t_0)\|_{L^\infty_{x,h}} = M$. At this stage we note that 
the maximum being attained at $(\bar x, \bar h)$ imposes an upper bound for $|\bar h|$.  
Indeed, for every $|h| \geq \xi_0$, since $0 \leq \xi \leq \xi_0$ we have
\[
|v(\cdot,\cdot;h)| \leq \frac{2 \|\theta\|_{L^\infty}}{|h|^\alpha} \leq \frac{2\|\theta_0\|_{L^\infty}}{\xi_0^\alpha} = \frac{M}{2}.
\]
This shows that we must have $|\bar h| \leq \xi_0$.

Using Lemmas \ref{lem:ODE} and \ref{lem:ubdd} we bound the right-side of \eqref{eq:ineq1}, and obtain that for $t\in(0,t_0]$ we have
\begin{align*}
L v^2+\frac{1}{(\xi^2+|h|^2)^\alpha} D_\gamma[\delta_h\theta]\leq
\frac{1}{2(\xi^2+|h|^2)^\alpha}D_\gamma[\delta_h\theta] + \frac{1}{4c_0 |h|^\gamma} v^2
\end{align*}
pointwise in $x,h \in \TT$, with $|h| \leq \xi_0$. 
On the other hand, the lower bound \eqref{eq:N2} on (a fourth of) the dissipation entails
\begin{align*}
L v^2+\frac{1}{4c_0|h|^\gamma} \left[\frac{|v|}{\|v\|_{L^\infty_{x,h}}}\right]^{\frac{\gamma}{1-\alpha}}v^2+
\frac{1}{4(\xi^2+|h|^2)^\alpha} D_\gamma[\delta_h\theta]\leq
 \frac{1}{4c_0 |h|^\gamma} v^2.
\end{align*}
Consequently, for $t\in(0,t_0]$, by again using \eqref{eq:N2}, we have
\begin{equation}\label{eq:ineq2}
L v^2+\frac{1}{4c_0|h|^\gamma} \left(\left[\frac{|v|}{\|v\|_{L^\infty_{x,h}}}\right]^{\frac{\gamma}{1-\alpha}}-1\right)v^2+
\frac{1}{4c_0 |h|^\gamma}\left[\frac{|v|}{\|v\|_{L^\infty_{x,h}}}\right]^{\frac{\gamma}{1-\alpha}} v^2\leq 0
\end{equation}
pointwise in $x,h \in \TT$, with $|h| \leq \xi_0$. 

Let $t \in [t_0-\epsilon, t_0)$ be arbitrary, where $\epsilon >0$ is small enough so that $\|v(t)\|_{L^\infty_{x,h}} \geq 3 M/4$ for all $t$ in this interval.
In particular, this ensures that the maximum of $|v(x,t;h)|$ cannot be attained at an $h$ with $|h| \geq \xi_0$.
For such $t$ close to $t_0$, we evaluate estimate \eqref{eq:ineq2} above at a point $(\bar x, \bar h) = (\bar x(t), \bar h(t)) \in \TT\times \TT$ at which $v^2(t)$ attains its maximum value of $M$.
Since, at that point, $\de_hv^2=\de_xv^2=0$, $\Lambda^\gamma v^2\geq 0$, $|v(\bar x, t ; \bar h)|=\|v(t)\|_{L^\infty_{x,h}}  $, and $|\bar h| \leq \xi_0$, we arrive at
\begin{align*}
(\de_tv^2)(\bar x,t; \bar h)+\frac{(3M/4)^2}{4c_0 \xi_0^\gamma} 
\leq Lv^2(\bar x,t; \bar h) + \frac{ v^2(\bar x, t; \bar h)}{4c_0 |\bar h|^\gamma}  \leq 0.
\end{align*}
Here we used that the second term on the left of \eqref{eq:ineq2} vanishes at $(\bar x, \bar h)$ since $|v(\bar x, t; \bar h)| \|v(t)\|_{L^\infty_{x,h}}^{-1} = 1$.
Consequently,
\begin{equation}\label{eq:ineq3}
(\de_tv^2)(\bar x,t; \bar h) < - \frac{9M^2}{64c_0 \xi_0^\gamma} 
\end{equation}
for all $t\in [t_0-\epsilon,t_0)$.
Following an argument in \cite{CTV13}*{Appendix B}, one may show that for almost every $t$ in $[t_0-\epsilon,t_0)$ we have
\begin{align*}
\frac{\d}{\d t}\| v(t)\|_{L^\infty_{x,h}}^2\leq (\de_tv^2)(\bar x, t;\bar h) <  - \frac{9M^2}{64c_0 \xi_0^\gamma} .
\end{align*}
from which it follows upon using the fundamental theorem of calculus that $\|v(t_0)\|_{L^\infty_{x,h}} < M$.
We may thus conclude that $t_0=\infty$, or in other words
\begin{align*}
\|v(t)\|_{L^\infty_{x,h}}\leq M, \qquad\forall t\geq 0. 
\end{align*}
Notice that $\xi(t)\equiv 0$ for all $t\geq T_\star$, where $T_\star$ is given by \eqref{eq:regtime}. Hence,
\begin{align*}
[\theta(t)]_{C^\alpha}=\|v(t)\|_{L^\infty_{x,h}}\leq M = \frac{4 \|\theta_0\|_{L^\infty}}{\xi_0^\alpha}, \qquad\forall t\geq T_\star, 
\end{align*}
and the proof is completed.


\section{A lower bound for time of local existence}\label{sec:local}
In this section, we explicitly compute a lower bound on the local time of existence of solutions to 
the supercritical SQG equation. As mentioned earlier, such a time will depend on norms which
are not scaling-critical. Precisely, we have the following result.

\begin{proposition}
\label{prop:local}
Let  $\theta_0 \in H^2$  be given, and consider the  unique local in time solution  of the supercritical SQG equation~\eqref{eq:SQG:gamma}
\[ 
\theta \in L^\infty(0,T_1;H^2) \cap L^2(0,T_1;H^{2+\gamma/2})
\] 
originating from $\theta_0$. There exists a universal constant $C_0 > 0$ such that the lower bound
\begin{align}
T_1 \geq \frac{1}{C_0 \|\theta_0\|_{L^2}^{\gamma/2} \|\theta_0\|_{\dot{H}^2}^{2-\gamma/2}}
\label{eq:time:local}
\end{align}
holds. 
\end{proposition}

Before giving the proof of \eqref{eq:time:local}, we recall a number of useful inequalities involving the fractional Laplacian.
We recall the Gagliardo-Nirenberg inequality
\[
\| f\|_{L^q}\leq C_{q} \|\Lambda^{1-\frac2q}f\|_{L^2},
\]
valid for $q\in [2,\infty)$  and mean zero functions $f$.
Two particularly useful cases are
\begin{equation}\label{eq:Lgamma}
\| f\|_{L^{4/\gamma}}\leq C_\gamma \|\Lambda^{1-\gamma/2}f\|_{L^2}, \qquad 
\mbox{and} \qquad 
\| f\|_{L^{4/(2-\gamma)}}\leq C_\gamma \|\Lambda^{\gamma/2}f\|_{L^2},
\end{equation}
where the constant $C_\gamma$ is bounded uniformly from above  for $\gamma \in [\gamma_0,3/2]$, so that the
dependence on $\gamma$ will be dropped. We will make use of the interpolation inequality
\begin{equation}\label{eq:interp}
\|f\|^2_{\dot{H}^\sigma}\leq C \|f\|_{L^2}^{2-\sigma} \|f\|_{\dot{H}^2}^\sigma,
\end{equation}
valid for $\sigma\in[0,2]$, with constant $C$ independent of $\sigma$. Lastly we shall use that $\|\RR^\perp\|_{L^p \to L^p} \leq C p$ for $p\geq 2$, with $C>0$ a universal constant. In particular we apply this bound for $p = 4/\gamma$ and $p = 4/(2-\gamma)$ and in this case the operator norm of $\RR^\perp$ on $L^p$ is bounded independently of $\gamma \in [\gamma_0 ,3/2]$.

\begin{proof}[Proof of Proposition~\ref{prop:local}]
 The existence of such a solution $\theta$ on a maximal time interval $[0,T_1)$ follows e.g. from~\cite{Ju07}. The proof of the proposition consists of an a priori $\dot{H}^2$ estimate. First, recall that since $\nabla \cdot \u = 0$ we immediately  have
\begin{align}
\|\theta(t)\|_{L^2}^2 \leq \|\theta_0\|_{L^2}^2 - \int_0^t \|\Lambda^{\gamma/2} \theta(s)\|_{L^2}^2 \d s \leq \|\theta_0 \|_{L^2}^2.
\label{eq:L2:ode}
\end{align}
Taking an inner product of \eqref{eq:SQG:gamma} with $\Delta^2 \theta$, using that $\nabla \cdot \u = 0$, and the bounds \eqref{eq:Lgamma} and \eqref{eq:interp} we obtain
\begin{align}
\frac{1}{2} \frac{\d}{\d t} \|\theta\|_{\dot{H}^2}^2 + \| \theta \|_{\dot{H}^{2+\gamma/2}}^2 
&=- \int \Delta( \u \cdot \nabla \theta) \Delta \theta \,\d x = -\int \Delta \u \cdot \nabla \theta \Delta \theta\, \d x - 2 \int \nabla \u : \nabla^2 \theta \Delta \theta \,\d x  \notag\\
&\leq \| \Delta \u\|_{L^2} \|\nabla \theta\|_{L^{4/\gamma}} \| \Delta \theta\|_{L^{4/(2-\gamma)}}+ 2 \|\nabla \u\|_{L^{4/\gamma}} \|\nabla^2 \theta\|_{L^2} \|\Delta \theta\|_{L^{4/(2-\gamma)}} \notag \\
&\leq C \|\theta\|_{\dot{H}^2} \|\theta\|_{\dot{H}^{2-\gamma/2}} \|\theta\|_{\dot{H}^{2+\gamma/2}} \leq \frac{1}{2} \| \theta \|_{\dot{H}^{2+\gamma/2}}^2  + C \|\theta\|_{\dot{H}^2}^{4-\gamma/2}  \|\theta\|_{L^2}^{\gamma/2}.\label{eq:H2:ode}
\end{align}
Letting $y(t) = \|\theta(t)\|_{\dot{H}^2}$, from \eqref{eq:H2:ode} above and the $L^2$ maximum principle  \eqref{eq:L2:ode} it follows that 
\[
\dot{y} \leq A y^{3-\gamma/2} \qquad \mbox{where} \qquad A = C_0 \|\theta_0\|_{L^2}^{\gamma/2}
\]
and $C_0>0$ is a fixed universal constant. Solving the above ODE it follows that 
\begin{align}
y(t) \leq \frac{y_0}{(1 - (2-\gamma/2)A y_0^{2-\gamma/2} t)^{1/(2-\gamma/2)}}
\label{eq:H2:bound}
\end{align}
From \eqref{eq:H2:bound} it follows that the $H^2$ norm of $\theta$ does not blow before 
\[
T = \frac{1}{(2-\gamma/2) A y_0^{2-\gamma/2}} \geq \frac{1}{2 A y_0^{2-\gamma/2}} 
= \frac{1}{2 C_0 \|\theta_0\|_{L^2}^{\gamma/2} \|\theta_0\|_{\dot{H}^{2}}^{2-\gamma/2}} = T_1
\]
which concludes the proof.
\end{proof}

\section{Proof of Theorem \ref{thm:superGWP}}
\label{sec:theproof}
Given $\theta_0$ in $H^2$, by the local existence theorem (cf.~Proposition~\ref{prop:local}) we have that the solution of \eqref{eq:SQG:gamma} with initial datum $\theta_0$ does not blow up until 
\[
T_1 = \frac{1}{C_0 \|\theta_0\|_{L^2}^{\gamma/2} \|\theta_0\|_{\dot{H}^2}^{2-\gamma/2}}.
\]
On the other hand, by the eventual regularity theorem (cf.~Theorem~\ref{thm:eventual}) we know that after time
\[
T_\star = C_0 \alpha^{\frac{\gamma(2-\gamma)}{1-\gamma}} \left(\|\theta_0\|_{L^2}^{1/2} \|\theta_0\|_{\dot{H}^2}^{1/2} \right)^{\frac{\gamma}{1-\gamma}}.
\]
the solution remains smooth, where $\alpha \in (1-\gamma,1)$ is arbitrary. Here we used that in two dimensions we have the bound $\| \theta_0\|_{L^\infty} \leq C \|\theta_0\|_{L^2}^{1/2} \|\theta_0\|_{\dot{H}^2}^{1/2}$. Also $C_0 \geq 2$ is a universal constant.

The proof is concluded once we show that for $\gamma$ sufficiently close to $1$ we may choose a suitably small $\alpha \in (1-\gamma,1)$ such that 
\[
T_\star \leq T_1.
\]
This is equivalent to
\begin{align}
C_0^{-2} \alpha^{-\frac{\gamma(2-\gamma)}{1-\gamma}} 
\geq \| \theta_0 \|_{L^2}^{\frac{\gamma}{2(1-\gamma)} + \frac{\gamma}{2}} \|\theta_0\|_{\dot{H}^{2}}^{\frac{\gamma}{2(1-\gamma)} + 2- \frac{\gamma}{2}} 
= \| \theta_0 \|_{L^2}^{\frac{\gamma(2-\gamma)}{2(1-\gamma)}} \|\theta_0\|_{\dot{H}^{2}}^{\frac{(2-\gamma)^2}{2(1-\gamma)}}.
\label{eq:to:do:*:1}
\end{align}
Assuming that
\[
\|\theta_0\|_{L^2}^{\frac{\gamma}{2}} \|\theta_0\|_{\dot{H}^2}^{\frac{2-\gamma}{2}} \leq R 
\]
it follows by raising both sides to the power $(2-\gamma)/(1-\gamma)$ that 
\[
 \| \theta_0 \|_{L^2}^{\frac{\gamma(2-\gamma)}{2(1-\gamma)}} \|\theta_0\|_{\dot{H}^{2}}^{\frac{(2-\gamma)^2}{2(1-\gamma)}} 
 \leq R^{\frac{2-\gamma}{1-\gamma}} 
\]
and thus \eqref{eq:to:do:*:1} holds if we choose $\alpha$ such that 
\begin{align}
C_0^{-2} \alpha^{-\frac{\gamma(2-\gamma)}{1-\gamma}} 
\geq R^{\frac{2-\gamma}{1-\gamma}} 
\qquad 
\Leftrightarrow  
\qquad 
R^{-\frac{1}{\gamma}} C_0^{- \frac{2(1-\gamma)}{\gamma(2-\gamma)}} \geq \alpha 
\label{eq:to:do:*:2}.
\end{align}
To conclude,
we let
\[
\alpha = 
\min \left \{  2(1-\gamma), \frac 12 \right \}
\]
which combined with \eqref{eq:to:do:*:1}--\eqref{eq:to:do:*:2} imply that there exists $\gamma_1 = \gamma_1(R) \in [\gamma_0,1)$, such that for all $\gamma \in [\gamma_1 , 1)$ $T_\star \leq T_1$. This shows that the solution cannot blow up in finite time, concluding the proof.

\section*{Acknowledgements}
The work of VV was in part supported by the NSF grant DMS-1348193.




\begin{bibdiv}
\begin{biblist}

\bib{CV10a}{article}{
   author={Caffarelli, Luis A.},
   author={Vasseur, Alexis},
   title={Drift diffusion equations with fractional diffusion and the
   quasi-geostrophic equation},
   journal={Ann. of Math. (2)},
   volume={171},
   date={2010},
   pages={1903--1930},
}

\bib{CL03}{article}{
   author={Chae, Dongho},
   author={Lee, Jihoon},
   title={Global well-posedness in the super-critical dissipative quasi-geostrophic equations},
   journal={Comm. Math. Phys.},
   volume={233},
   date={2003},
   pages={297--311},
}

\bib{CCS10}{article}{
   author={Chan, Chi Hin},
   author={Czubak, Magdalena},
   author={Silvestre, Luis},
   title={Eventual regularization of the slightly supercritical fractional
   Burgers equation},
   journal={Discrete Contin. Dyn. Syst.},
   volume={27},
   date={2010},
   pages={847--861},
}
	

\bib{CMZ07}{article}{
   author={Chen, Qionglei},
   author={Miao, Changxing},
   author={Zhang, Zhifei},
   title={A new Bernstein's inequality and the 2D dissipative
   quasi-geostrophic equation},
   journal={Comm. Math. Phys.},
   volume={271},
   date={2007},
   pages={821--838},
}


\bib{CCW00}{article}{
   author={Constantin, Peter},
   author={C\'ordoba, Diego},
   author={Wu, Jiahong},
   title={On the critical dissipative quasi-geostrophic equation},
   journal={Indiana Univ. Math. J.},
   volume={50},
   date={2001},
   pages={97--107},
}
	
	

\bib{CTV13}{article}{
   author={Constantin, Peter},
   author={Tarfulea, Andrei},
   author={Vicol, Vlad},
   title = {Long time dynamics of forced critical SQG},
   journal={ArXiv 1308.0640, Comm. Math. Phys., to appear},
   date = {2014},
}

\bib{CV12}{article}{
   author={Constantin, Peter},
   author={Vicol, Vlad},
   title={Nonlinear maximum principles for dissipative linear nonlocal
   operators and applications},
   journal={Geom. Funct. Anal.},
   volume={22},
   date={2012},
   pages={1289--1321},
}

\bib{CW08}{article}{
    author={Constantin, Peter},
    author={Wu, Jiahong},
    title={Regularity of {H}\"older continuous solutions of the supercritical quasi-geostrophic equation},
    journal={Ann. Inst. H. Poincar\'e Anal. Non Lin\'eaire},
    volume={25},
    date={2008},
    pages={1103--1110},
}

\bib{CC04}{article}{
   author={C{\'o}rdoba, Antonio},
   author={C{\'o}rdoba, Diego},
   title={A maximum principle applied to quasi-geostrophic equations},
   journal={Comm. Math. Phys.},
   volume={249},
   date={2004},
   pages={511--528},
}


\bib{Dab11}{article}{
   author={Dabkowski, Michael},
   title={Eventual regularity of the solutions to the supercritical dissipative quasi-geostrophic equation},
   journal={Geom. Funct. Anal.},
   volume={21},
   date={2011},
   pages={1--13},
}


\bib{DKSV14}{article}{
   author={Dabkowski, Michael},
   author={Kiselev, Alexander},
   author={Silvestre, Luis},
   author={Vicol, Vlad},  
   title={Global well-posedness of slightly supercritical active scalar equations},
   journal={Analysis and PDE},
   volume={7},
   date={2014},
   pages={43--72},
}



\bib{Dong10}{article}{
	author = {Dong, Hongjie},
	title = {Dissipative quasi-geostrophic equations in critical Sobolev spaces: smoothing effect and global well-posedness},
	journal = {Discrete Contin. Dyn. Syst.},
	volume = {26},
	date = {2010},
	number = {4},
	pages = {1197--1211},
}

\bib{DDL09}{article}{
   author={Dong, Hongjie},
   author={Du, Dapeng},
   author={Li, Dong},
   title={Finite time singularities and global well-posedness for fractal
   Burgers equations},
   journal={Indiana Univ. Math. J.},
   volume={58},
   date={2009},
   pages={807--821},
}
	

\bib{DP09}{article}{
	author={Dong, Hongjie},
	author={Pavlovic, Natasa},
	title={Regularity Criteria for the Dissipative Quasi-Geostrophic Equations in H\"older Spaces},
	journal={Commun. Math. Phys.},
	volume={290},
	date={2009},
	pages={801--812},
}

\bib{Ju07}{article}{
   author={Ju, Ning},
   title={Dissipative 2D quasi-geostrophic equation: local well-posedness, global regularity and similarity solutions},
   journal={Indiana Univ. Math. J.},
   volume={56},
   date={2007},
   pages={187--206},
}


\bib{Kis11}{article}{
   author={Kiselev, Alexander},
   title={Nonlocal maximum principles for active scalars},
   journal={Adv. Math.},
   volume={227},
   date={2011},
   pages={1806--1826},
}

\bib{KN09}{article}{
   author={Kiselev, A.},
   author={Nazarov, F.},
   title={A variation on a theme of Caffarelli and Vasseur},
   journal={Zap. Nauchn. Sem. S.-Peterburg. Otdel. Mat. Inst. Steklov. (POMI)},
   volume={370},
   date={2009},
   pages={58--72, 220},
}


\bib{KNS08}{article}{
   author={Kiselev, Alexander},
   author={Nazarov, Fedor},
   author={Shterenberg, Roman},
   title={Blow up and regularity for fractal Burgers equation},
   journal={Dyn. Partial Differ. Equ.},
   volume={5},
   date={2008},
   pages={211--240},
}

\bib{KNV07}{article}{
   author={Kiselev, A.},
   author={Nazarov, F.},
   author={Volberg, A.},
   title={Global well-posedness for the critical 2D dissipative quasi-geostrophic equation},
   journal={Invent. Math.},
   volume={167},
   date={2007},
   pages={445--453},
}

\bib{HK07}{article}{
   author={Hmidi, Taoufik},
   author={Keraani, Sahbi},
   title={Global solutions of the super-critical 2D quasi-geostrophic
   equation in Besov spaces},
   journal={Adv. Math.},
   volume={214},
   date={2007},
   pages={618--638},
}


\bib{Miu06}{article}{
   author={Miura, Hideyuki},
   title={Dissipative quasi-geostrophic equation for large initial data in
   the critical Sobolev space},
   journal={Comm. Math. Phys.},
   volume={267},
   date={2006},
   pages={141--157},
}


\bib{Sil10a}{article}{
   author={Silvestre, Luis},
   title={Eventual regularization for the slightly supercritical
   quasi-geostrophic equation},
   journal={Ann. Inst. H. Poincar\'e Anal. Non Lin\'eaire},
   volume={27},
   date={2010},
   pages={693--704},
}

\bib{SVZ13}{article}{
   author={Silvestre, Luis},
   author={Vicol, Vlad},
   author={Zlatos, Andrej},
   title={On the Loss of Continuity for Super-Critical Drift-Diffusion Equations},
   journal={Arch. Ration. Mech. Anal.},
   volume={207},
   date={2013},
   pages={845--877},
}


\bib{Wu05}{article}{
   author={Wu, Jiahong},
   title={Global solutions of the 2D dissipative quasi-geostrophic equation in Besov spaces},
   journal={SIAM J. Math. Anal.},
   volume={36},
   date={2004/05},
   pages={1014--1030},
}

\bib{Wu07}{article}{
   author={Wu\ ,Jiahong},
   title={Existence and uniqueness results for the 2-D dissipative quasi-geostrophic equation},
   journal={Nonlinear Analysis},
   volume={67},
   date={2007},
   pages={3013--3036},
}


\bib{Yu08}{article}{
   author={Yu, Xinwei},
   title={Remarks on the global regularity for the super-critical 2D
   dissipative quasi-geostrophic equation},
   journal={J. Math. Anal. Appl.},
   volume={339},
   date={2008},
   pages={359--371},
}

\end{biblist}
\end{bibdiv}
\end{document}